\pdfoutput=1
\documentclass[a4paper,final,fleqn]{article}

\usepackage[margin=3cm]{geometry}

\usepackage[lm,grey]{janm}
\usepackage{math}

\usepackage{marginnote}
\usepackage{xspace}

\usepackage[numbers]{natbib}
\usepackage{showkeys}

\newcommand{\php}{\ph_{\!+}}
\newcommand{\phm}{\ph_{\!-}}

\newcommand{\ld}{\text{\tiny\ensuremath{\bullet}}}

\newcommand{\Jt}{J}

\newcommand{\vs}{\varsigma}

\newcommand{\smat}[1]{\left(\begin{smallmatrix}#1\end{smallmatrix}\right)}

\DeclareMathOperator{\spec}{spec}

\newcommand{\mkdv}{\text{mKdV}\xspace}
\newcommand{\dnls}{\text{dNLS}\xspace}

\newcommand{\gradz}{\partial}
\newcommand{\gradh}{\nabla}
\newcommand{\nbr}[1]{\{#1\}}
\newcommand{\gbr}[1]{\{#1\}_{\partial_{x}}}

\newcommand{\xh}{\mathfrak{x}}
\newcommand{\yh}{\mathfrak{y}}
\newcommand{\xz}{x}
\newcommand{\yz}{y}

\newcommand{\Gh}{G^{\mathrm{mKdV}}}
\newcommand{\Gz}{G}

\newcommand{\Hz}{S}
\newcommand{\Hh}{K}

\newcommand{\Vh}{V^{\mathrm{mKdV}}}
\newcommand{\Vz}{V}

\newcommand{\Wh}{W^{\mathrm{mKdV}}}
\newcommand{\Wz}{W}

\newcommand{\Uh}{U^{\mathrm{mKdV}}}
\newcommand{\Uz}{U}

\newcommand{\Dh}{\Dl_{\mathrm{mKdV}}}
\newcommand{\dDh}{\Dl^{\ld}_{\mathrm{mKdV}}}
\newcommand{\Dz}{\Dl}
\newcommand{\dDz}{\Dl^{\ld}}

\newcommand{\Fh}{F_{\mathrm{mKdV}}}
\newcommand{\Fz}{F}

\newcommand{\Oz}{\Om}
\newcommand{\Oh}{\Psi}

\newcommand{\Lh}{L_{\mathrm{mKdV}}}
\newcommand{\Lz}{L}

\newcommand{\Gmz}{\Gm}
\newcommand{\Gmh}{\Sg}

\newcommand{\Mz}{M}
\newcommand{\Mh}{M_{\mathrm{mKdV}}}

\newcommand{\gMz}{\grave{M}}
\newcommand{\gMh}{\grave{M}_{\mathrm{mKdV}}}

\renewcommand{\phm}{\ph_{1}}
\renewcommand{\php}{\ph_{2}}

\newcommand{\cph}{P\ph}

\newcommand{\rest}{\sharp}
\newcommand{\omk}{\om^{\kdv}}

\hypersetup{
  pdfinfo={
    Title={The mKdV and NLS hierarchies revisited},
    Author={Jan-Cornelius Molnar, Yannick Widmer},
    Subject={mKdV and NLS hierarchies, mKdV and NLS Birkhoff normal forms},
    Keywords={Hamiltonian hierarchies, modified Korteweg–de Vries, nonlinear Schrödinger, integrable PDEs}
  }
}

\title{The mKdV and NLS hierarchies revisited}
\author{Jan-Cornelius Molnar, Yannick Widmer}
\date{\today}

\begin{document}

\maketitle

\begin{abstract}
The purpose of this paper is to express the entire hierarchy of mKdV vector fields as restrictions of vector fields in the NLS hierarchy. The result is proved using the normal form theory of the two equations.



\end{abstract}

\section{Introduction}

We consider the \emph{defocusing \mkdv equation}
\[
  u_{t} = -u_{xxx} + 6u^{2}u_{x},
\]
on the circle $\T = \R/\Z$ with $u$ real-valued. The \mkdv equation can be viewed as a Hamiltonian PDE with Hamiltonian
\[
  \Hh(u) \defl \frac{1}{2}\int_{\T} (u_{x}^{2} + u^{4})\,\dx,
\]
on the standard Sobolev space $H^m_{r} \defl H^{m}(\T,\R)$, $m\ge 0$, as phase space endowed with the Poisson bracket proposed by Gardner
\begin{equation}
  \label{mkdv-poi}
  \gbr{F,G} \defl \int_{\T} \gradh_{u} F\,\partial_x\gradh_{u} G\, \dx.
\end{equation}
Here, $\gradh_{u} F$ denotes the $L^{2}$-gradient of a $C^{1}$-functional on $H_{r}^{0}$. The mean value $[u] \defl \int_{\T} u\,\dx$ is a Casimir for the Gardner bracket and hence is preserved by the \mkdv flow. The defocusing \mkdv equation then takes the form $u_{t} = \gbr{u,\Hh}$. This equation admits an infinite sequence of recursively defined pairwise Poisson commuting integrals referred to as \emph{\mkdv hierarchy},
\begin{align*}
  \Hh_{1}(u) = \frac{1}{2}\int_{\T} u^{2}\,\dx,\qquad
  \Hh_{2}(u) = \Hh(u) = \frac{1}{2}\int_{\T} (u_{x}^{2} + u^{4})\,\dx,\qquad
  \dotsc.
\end{align*}
Each of these Hamiltonians leads to a Hamiltonian PDE.


It is well known that the mKdV equation, being closely related to the KdV equation via the Miura map~\cite{Miura:1968uq}, is also closely related to the \emph{\nls system}
%
\begin{equation}
\label{nls-sys}
\begin{split}
  \ii \partial_{t}\phm &= \phantom{-}\gradz_{\php}\Hz = -\partial_{xx}\phm + 2\php\phm^{2},\\
  \ii \partial_{t}\php &= -\gradz_{\phm}\Hz = \phantom{-}\partial_{xx}\php - 2\phm\php^{2}.
\end{split}
\end{equation}
This system can be viewed as a Hamiltonian PDE with Hamiltonian
\begin{equation}
  \label{nls-ham}
  \Hz(\ph) = \int_{\T} (\partial_{x}\php\partial_{x}\phm + \php^{2}\phm^{2}) \,\dx,
\end{equation}
on the phase space $\Hc^{m}_{c} \defl H^{m}_{c}\times H^{m}_{c}$, $m\ge 0$, with Poisson bracket
\begin{equation}
  \label{nls-poi}
  \nbr{F,G} 
  \defl 
  -\ii\int_{\T} 
  (\gradz_{\phm} F\, \gradz_{\php} G 
    - \gradz_{\php}F\,\gradz_{\phm} G)\,\dx.
\end{equation}
Here $H_{c}^{m}$ denotes the standard Sobolev space $H^{m}(\T,\C)$, $\phm$, $\php$ denote the two components of $\ph\in\Hc^{m}_{c}$, and $\gradz_{\phm} F$, $\gradz_{\php} F$ denote the two components of the $\Lc^2$-gradient $\gradz F$ of a $C^1$-functional $F$ on $\Hc^{0}_{c}$.

The Hamiltonian $\Hz(\ph)$ admits for any $m\ge 0$ the invariant real subspaces
\[
  \Hc_{r}^{m} \defl \setdef{\ph\in \Hc_{c}^{m}}{\php =  \ob{\phm}},\quad
  \Hc_{i}^{m} \defl \setdef{\ph\in \Hc_{c}^{m}}{\php = -\ob{\phm}}.
\]
When \eqref{nls-sys} is restricted to $\Hc_{r}^{m}$, with $\ph = (v,\ob{v})$, one obtains the \emph{defocusing \nls} (\dnls) equation
\[
  \ii \partial_{t}v = \ii\nbr{v,\Hz} = -\partial_{xx}v + \abs{v}^{2}v ,\qquad \Hz(v,\ob{v}) = \int_{\T} (\abs{v_{x}}^{2} + \abs{v}^{4}) \,\dx.
\]
Similarly, when \eqref{nls-sys} is restricted to $\Hc_{i}^{m}$, with $\ph = (\ii v,\ii \ob{v})$, one obtains the \emph{focusing \nls} equation
\[
  \ii \partial_{t}v = \ii\nbr{v,\Hz} = -\partial_{xx}v - \abs{v}^{2}v ,\qquad
  \Hz(\ii v,\ii \ob{v}) = -\int_{\T} (\abs{v_{x}}^{2} - \abs{v}^{4}) \,\dx.
\]

The \nls system \eqref{nls-sys} also admits an infinite sequence of recursively defined pairwise Poisson commuting integrals referred to as \emph{\nls hierarchy}\footnote{
In comparison with \cite{Grebert:2014iq}, the $n$th Hamiltonian of the \nls hierarchy (for $n\ge 2$) is multiplied by $(-\ii)^{n+1}$ to make the corresponding Hamiltonian flow real-valued for real-valued $\ph$.
}, $\Hz_{1}(\ph) = \int_{\T} \phm\php  \,\dx$,
\begin{align*} 
  \Hz_{2}(\ph) &= \frac{\ii}{2} \int_{\T} (\php\partial_{x}\phm - \phm\partial_{x}\php)\,\dx,\\
   \Hz(\ph) = \Hz_{3}(\ph)
    &= \;\;\phantom{\ii}\int_{\T} (\partial_{x}\phm\partial_{x}\php + \phm^{2}\php^{2}) \,\dx,\\
  \Hz_{4}(\ph) &= \;\;\ii\int_{\T} (\phm\partial_{xxx}\php - 3\phm^{2}\php\partial_{x}\php) \,\dx,
  \qquad
 \ldots
\end{align*}
The Hamiltonian $\Hz_{4}$ gives rise to the system
\begin{equation}
\label{mkdv-sys}
\begin{split}
  \partial_{t} \phm &= \nbr{\phm,\Hz_{4}}
  					 = -\ii\gradz_{\php}\Hz_{4}
					 = -\partial_{xxx}\phm + 6\phm\php\partial_{x}\phm,\\
  \partial_{t} \php &= \nbr{\php,\Hz_{4}}
  					 = \phantom{-}\ii\gradz_{\phm}\Hz_{4}
					 = -\partial_{xxx}\php + 6\php\phm\partial_{x}\php.
\end{split}
\end{equation}
This system admits, for any $m\ge 1$, the real invariant subspaces
\[
  \Ec_{r}^{m} \defl \setdef{\ph \in \Hc_{r}^m}{\php = \phm},\qquad
  \Ec_{i}^{m} \defl \setdef{\ph \in \Hc_{i}^m}{\php = \phm}.
\]
When \eqref{mkdv-sys} is restricted to $\Ec_{r}^{m}$, with $\ph = (u,u)$ and $u$ real-valued, one obtains the defocusing \mkdv equation
\[
  \partial_{t}u
   = -\ii \gradz_{\php}\Hz_{4}\big|_{(u,u)}
   = \partial_{x}\gradh\Hh_{2}(u)
   = -\partial_{xxx}u + 6u^{2}\partial_{x} u.
\]
Similarly, when \eqref{mkdv-sys} is restricted to $\Ec_{i}^{m}$, with $\ph = (\ii u,\ii u)$ and $u$ real-valued, one obtains the focusing \mkdv equation
\[
  \ii \partial_{t}u
   = -\ii \gradz_{\php}\Hz_{4}\big|_{(\ii u,\ii u)}
   = \partial_{x}\gradh\Hh_{2}(\ii u)
   = \ii(-\partial_{xxx}u - 6u^{2}\partial_{x} u).
\]

The main purpose of this paper is to show that in the same way the entire \mkdv hierarchy is contained in the \nls hierarchy. By a slight abuse of notation, we identify $\Ec_{c}^{m} \defl \setdef{\ph\in \Hc_{c}^{m}}{\php=\phm}$ with $H_{c}^{m}$ and denote the restriction of a functional $F\colon \Hc_{c}^{m}\to \C$ to $H_{c}^{m}$ by $F^{\rest}\colon H_c^m\to \C$, $u \mapsto F(u,u)$. The Hamiltonian vector field of $F$ with respect to the Poisson bracket~\eqref{nls-poi} is denoted by
\[
  X_{F} = -\ii J \gradz F,\qquad J = \mat[\bigg]{0 & 1\\-1 & 0},
\]
and its restriction to $H_{c}^{m}$ is denoted by $X_{F}^{\rest}= -\ii J (\gradz F)^{\rest}$. Similarly, for a function $G$ on $H_{c}^{m}$ we denote its Hamiltonian vector field with respect to the Gardner bracket by
\[
  Y_{G} = \partial_{x} \gradh G.
\]

\begin{thm}
\label{thm:ham-mkdv-nls}
For every $m\ge 1$ the Hamiltonian vector field $Y_{\Hh_{m}}$ of the \mkdv hierarchy and the Hamiltonian vector field $X_{\Hz_{2m}}$ of the \nls hierarchy satisfy
\[
  X_{\Hz_{2m}}^{\rest} = (Y_{\Hh_{m}},Y_{\Hh_{m}})
  \text{ on } H_{c}^{m}.
\]
In addition, $S_{2m}^{\rest} = 0$ on $H_{c}^{m}$ for every $m\ge 1$.\fish
\end{thm}

Loosely speaking, the above theorem says that each PDE of the \mkdv hierarchy can be viewed as a subsystem of the corresponding PDE in the \nls hierarchy. The defocusing and focusing cases of the PDEs are obtained by restriction to the subspaces $E_{r}^{m} = \setdef{u\in H_{c}^{m}}{u\text{ real-valued}}$ and $E_{i}^{m} = \setdef{\ii u\in H_{c}^{m}}{u\text{ real-valued}}$, respectively. Since the defocusing \mkdv equation on the circle can be identified with the \kdv equation (c.f. e.g.  \cite{Kappeler:2005gt,Kappeler:2006fr}), the \kdv equation on the circle as well can  be viewed as a subsystem of the \dnls equation.

A key ingredient into the proof of Theorem~\ref{thm:ham-mkdv-nls} are the following symmetries of the gradients of the Hamiltonians in the \nls hierarchy.

\begin{thm}
\label{thm:nls-sym}
(i) For every $\ph\in \Hc_{c}^{k-1}$ with $k\ge 1$, and any real $\al$,
\begin{align*}
  \gradz\Hz_{k}(\ph) &= (-1)^{k-1}P\gradz\Hz_{k}(\cph),
  &P &=
  \smat{ & 1 \\ 1 & },\\
  \gradz\Hz_{k}(\ph) &= R_{\al}\gradz\Hz_{k}(R_{\al}\ph),
  &R_{\al} &= \smat{\e^{\ii \al} & \\  & \e^{-\ii \al}}.
\end{align*}
In particular, $X_{\Hz_{k}}(R_{\al}\ph) = R_{\al} X_{\Hz_{k}}(\ph)$ and $X_{\Hz_{k}}(P\ph) = (-1)^{k}PX_{\Hz_{k}}(\ph)$.

(ii)
If $\ph\in \Hc_{c}^{2m-1}$, $m\ge 1$, with $\cph = R_{\al}\ph$ for some real $\al$, then
\[
  -\ii J \gradz \Hz_{2m}(\ph) = \partial_{x} \gradz \Hz_{2m-1}(\ph).\fish
\]
\end{thm}

As an application we compare the solution curves of the defocusing \mkdv and the \dnls Hamiltonian vector fields in \emph{Birkhoff coordinates}. The defocusing \mkdv equation admits global Birkhoff coordinates $(\xh_{n},\yh_{n})_{n\ge 1}$ constructed in terms of global action-angle coordinates $(J_{n},\vt_{n})_{n\ge 1}$ -- see \cite{Kappeler:2008fk}. To give a precise definition, we introduce for any $m\ge 0$ the model space $h_{\star}^{m} = \ell_{m+1/2}^{2}(\N)\times\ell_{m+1/2}^{2}(\N)$ with elements $(\xh,\yh) =(\xh_{n},\yh_{n})_{n\ge1}$, where for any $\mathbb{A}\subset\Z$
\[
  \ell_{\al}^{2}(\mathbb{A})
   \defl \setdef[\bigg]{z \in \ell^{2}(\mathbb{A},\R)}
                   {\sum_{n\in \mathbb{A}} (1+\abs{n}^{2\al}) z_{n}^{2} < \infty},
\]
and endow this space with the Poisson structure $\{\xh_{n},\yh_{n}\} = -\{\yh_{n},\xh_{n}\} = 1$ while all other brackets vanish. The \emph{\mkdv Birkhoff map}
\[
  \Oh \colon H^{1}_{r} \to h^{1}_{\star}\times \R,\qquad
  u \mapsto ((\xh_{n},\yh_{n})_{n\ge 1},[u])
\]
defines a bi-real-analytic, canonical diffeomorphism, which transforms every Hamiltonian of the \mkdv hierarchy, on Sobolev spaces of the appropriate order, into \emph{Birkhoff normal form}, that is $\Hh_{m}\circ\Oh^{-1}$ is a real analytic function of the actions $J_{n} = (\xh_{n}^{2}+\yh_{n}^{2})/2$ and the average alone. In these coordinates, the Hamiltonian system with Hamiltonian $\Hh_{m}$ takes the particularly simple form
\[
  \dot \xh_{n} =  \eta_{n,m} \yh_{n},\quad
  \dot \yh_{n} =  -\eta_{n,m} \yh_{n},\quad
  \eta_{n,m} \defl \partial_{J_{n}} \Hh_{m} = \gbr{\Hh_{m},\vt_{n}},
\]
where $\eta_{n,m}$ is called the $n$th frequency of the Hamiltonian $\Hh_{m}$.

For the \dnls equation global Birkhoff coordinates $(\xz_{n},\yz_{n})_{n\in\Z}$ can be constructed in terms of global action-angle coordinates $(I_{n},\th_{n})_{n\in\Z}$ -- see \cite{McKean:1997ka,Grebert:2014iq} and references therein. As the model space we choose the Hilbert space $h_{r}^{m} = \ell_{m}^{2}(\Z)\times \ell_{m}^{2}(\Z)$, $m\ge 0$, with elements $(\xz,\yz)= (\xz_{n},\yz_{n})_{n\in\Z}$, which is endowed with the Poisson structure $\nbr{x_{n},y_{n}} = -\nbr{y_{n},x_{n}} = -1$\footnote{Since we closely follow \cite{Kappeler:2008fk} for the \mkdv and \cite{Grebert:2014iq} for the \nls normal form, respectively, we did not change the signs of the Poisson brackets on the model spaces, hence they are opposite.}.
The \dnls Birkhoff map
\[
  \Oz \colon \Hc_{r}^{0} \to h_{r}^{0},\quad \ph \mapsto (\xz_{n},\yz_{n})_{n\in\Z},
\]
defines a bi-real-analytic, canonical diffeomorphism, which transforms every Hamiltonian of the \nls hierarchy, on Sobolev spaces of the appropriate order, into Birkhoff normal form, that is $\Hz_{m}\circ\Oz^{-1}$ is a real analytic function of the actions $I_{n} = (x_{n}^{2}+y_{n}^{2})/2$ alone. In these coordinates, the Hamiltonian system with Hamiltonian $\Hz_{m}$ is given by
\[
  \dot \xz_{n} =  -\om_{n,m} \yz_{n},\quad
  \dot \yz_{n} =   \om_{n,m} \xz_{n},\quad
  \om_{n,m} \defl \partial_{I_{n}} \Hz_{m} = \nbr{\th_{n},\Hz_{m}},
\]
where $\om_{n,m}$ is called the $n$th frequency of $\Hz_{m}$.

We obtain the following relation of the frequencies of the two hierarchies in Birkhoff coordinates.

\begin{thm}
\label{thm:freq}
For any $n\ge 1$ and $m\ge 1$, we have on $H_{r}^{m-1}$,
\[
  (-1)^m \om_{ n,2m}^{\rest} = \eta_{n,m}.\fish
\]
\end{thm}

\begin{rem}
Note that $\eta_{n,m} = \omk_{n,m-1}\circ B$ for any $n\ge 1$ and $m\ge 1$. Moreover, the symmetry $\om_{-n,2m}^{\rest} = -\om_{ n,2m}^{\rest}$ for any $n\ge 1$ and $m\ge 1$ was obtained in \cite{Grebert:2002wb}.
\end{rem}

Theorem~\ref{thm:freq} follows  from Theorem~\ref{thm:ham-mkdv-nls} and the following  relation of the Birkhoff coordinates of the \mkdv and \dnls equations.

\begin{thm}
\label{thm:bnf}
On $H_{r}^{1}$,

\begin{equivenum}
\item $I_{0}^{\rest}$ vanishes if and only if the average $[u]$ is zero, and for any $n\ge 1$, $I_{n}^{\rest}$ vanishes if and only if $J_{n}$ is zero. (Note that $I_{-n}^{\rest} = I_{n}^{\rest}$ for any $n\ge 1$.)

\item
Each $I_{n}^{\rest}$, $n\in\Z$, is a real analytic function of the actions $(J_{m})_{m\ge 1}$ and the average alone. Conversely, the average and each $J_{n}$, $n\ge 1$, are real analytic functions of the actions $(I_{n})_{n\in\Z}$ alone.

\item For any $n,m\ge 1$, one has $\gbr{\th_{m}^{\rest},J_{n}} = \dl_{m,n}$.
In particular, since $\gbr{\vt_{m},J_{n}} = -\dl_{mn}$, it follows that $\vt_{m} + \th_{m}^{\rest}$ is a function of the actions $(J_{n})_{n\ge 1}$ and the average alone.\fish
\end{equivenum}
\end{thm}

\textit{Method of proof.}
The $m$th Hamiltonian $\Hz_{m}$ of the \nls hierarchy and its gradient $\partial \Hz_{m}$ satisfy on $\Hc_{r}^{m-1}$, $m\ge 1$, the trace formulae
\[
  \frac{1}{2^{m-1}}S_{m} = \sum_{m\in\Z} I_{n,m},
  \qquad
  \frac{1}{2^{m-1}}\partial S_{m} = \sum_{m\in\Z} \partial I_{n,m},
\]
where $I_{n,m}$ denotes the \nls action on level $m$ which were introduced by McKean \& Vaninsky~\cite{McKean:1997ka}. The actions are defined in terms of spectral data of the \emph{Zakharov-Shabat operator}
\[
  \Lz(\ph) = \mat[\bigg]{\ii & 0\\ 0 & -\ii}\partial_{x} + \mat[\bigg]{0 & \php\\\phm & 0},
\]
which arises in the Lax-pair formulation of \nls. More to the point, they are defined as functions of the discriminant $\Dz(\lm,\ph)$ of the fundamental solution associated to $L(\ph)$. We prove several symmetries of $\Dz$ and its gradient $\partial\Dz$ under the transformations $\ph\mapsto P\ph$ and $\ph\mapsto R_{\al}\ph$, $\al\in\R$, from which we obtain corresponding symmetries of the actions $I_{n,m}$, $n\in\Z$, $m\ge 1$, and their gradient $\partial I_{n,m}$. This establishes Theorem~\ref{thm:nls-sym}.

In the Lax-pair formulation of \mkdv there arises the \emph{Hill operator}
\[
  \Lh(u) = -\partial_{x}^{2} + B(u),
\]
with the potential $B(u)$ given by the Miura map $B(u) = u_{x} + u^{2}$.
Theorem~\ref{thm:ham-mkdv-nls} now follows from an identification of the spectra of the operators $\Lz(u,u)$ and $\Lh(u)$.
It turns out that for any solution $f=(f_{1},f_{2})$ of $\Lz(u,u)f = \lm f$, the function $g = f_{1} + \ii f_{2}$ is a solution of $\Lh(u)g = \lm^{2}g$. This implies that the Floquet matrix $\gMz(\lm)$ of $\Lz(u,u)$ is conjugated to the Floquet matrix $\gMh(\mu)$ with $\mu = \lm^{2}$ -- see \cite{Chodos:1980du}. Hence the discriminants $\Dz(\lm)$ and $\Dh(\mu)$ coincide at $\mu = \lm^{2}$.
The Hamiltonian hierarchies can be obtained from the asymptotic expansions of the corresponding discriminants, which gives for any $m\ge 1$ the identities
\[
  \frac{1}{2}\Hz_{2m-1}^{\rest} = \Hh_{m}, \qquad \Hz_{2m}^{\rest} = 0.
\]
Since by Theorem~\ref{thm:nls-sym} (applied for $\ph = (u,u)$) we have
\[
  X_{\Hz_{2m}}^{\rest} = (Y_{\Hz_{2m-1}^{\rest}},Y_{\Hz_{2m-1}^{\rest}}),
\]
Theorem~\ref{thm:ham-mkdv-nls} follows immediately.

The construction of the Birkhoff coordinates is based on the one of action-angle variables. Since the action variables can be obtained from spectral data of the operators $\Lh(u)$ and $\Lz(\ph)$, respectively, the observed relation of the discriminants for $\ph = (u,u)$ allow us to derive Theorem~\ref{thm:bnf}.

\emph{Related work.}
Chodos observed in~\cite{Chodos:1980du} that the Floquet matrix $\grave{\Mz}(\lm)$ of $\Lz(u,u)$ is conjugated to the Floquet matrix $\grave{\Mh}(\mu)$ of $\Lh(u)$ with $\mu = \lm^{2}$. He uses this to obtain the identity
\begin{equation}
  \label{chod}
  \Hh_{m} = \frac{1}{2}\Hz_{2m-1}^{\rest},\qquad m\ge 0,
\end{equation}
on the Sobolev spaces of the appropriate order, by realizing the \nls and \mkdv Hamiltonians as traces of certain powers of the operators $\Lz(u,u)$ and $\Lh(u)$, respectively. His approach, however, seems not to be suited to compare the Hamiltonian vector fields of the \nls and the \mkdv hierarchies, which is necessary to identify the PDEs in the \mkdv hierarchy as subsystems of the \nls hierarchy. Note that Theorem~\ref{thm:ham-mkdv-nls} does not follow immediately from \eqref{chod} by differentiation. Indeed, the indices of the identity $\Hh_{m} = \frac{1}{2}\Hz_{2m-1}^{\rest}$ for the Hamiltonians themselves are different from the indices of the identity $Y_{\Hh_{m}} = X_{\Hz_{2m}}^{\rest}$ for the Hamiltonian vector fields. This is due to the fact that the Poisson structure~\eqref{mkdv-poi} of \mkdv involves an additional derivative $\partial_{x}$ in comparison to the Poisson structure~\eqref{nls-poi} of \nls.

\citet{Dickey:2003tm} shows several algebraic relations of the NLS hierarchy and derives the mKdV hierarchy from the former by the method of Drinfeld-Sokolev reduction. However, the obtained relations are implicit in contrast to the explicit formulas given in Theorem~\ref{thm:ham-mkdv-nls} \& \ref{thm:nls-sym}.

Item (ii) of Theorem~\ref{thm:nls-sym} has been obtained by Magri~\cite{Magri:1978gh} in the case of a $C_{0}^{\infty}$-potential on $[0,1]$.
In this case, the NLS system~\eqref{nls-sys} can be written in Bi-Hamiltonian form
\begin{equation}
  \label{bi-ham}
  \partial_{t} \ph = K \partial S_{3},\qquad\text{and}\qquad \partial_{t} \ph = K_{2} \partial S_{2}.
\end{equation}
Here, $K = -\ii J$ denotes the standard Poisson structure and $K_{2}$ denotes the second Poisson structure
\[
  K_{2}f = \partial_{x}Pf - \ii R\ph\p*{\int_{0}^{x} (f_{1}\ph_{1}-f_{2}\ph_{2})\,\dx 
  + \int_{1}^{x} (f_{1}\ph_{1}-f_{2}\ph_{2})\,\dx},
\]
with Poisson bracket $\{F,G\}_{2} = \int_{\T} \partial F K_{2} \partial G\,\dx$. Both Poisson structures are compatible on $C_{0}^{\infty}$ in the sense that
\[
  \{F,G\}_{\lm} \defl \{F,G\} - \lm \{F,G\}_{2}
\]
is a Poisson bracket for any real $\lm$, which due to the nonlinear nature of the Jacobi identity is a nontrivial constraint. The second Poisson structure $K_{2}$ is non-constant. Furthermore, one obtains the following more general version of Theorem~\ref{thm:nls-sym} (ii)
\begin{equation}
  \label{ext-magri}
  K\partial \Hz_{m+1} = K_{2} \partial \Hz_{m},\qquad m\ge 1.
\end{equation}
Note that one has $K_{2}\big|_{\Ec_{c}} = \partial_{x}$, hence item (ii) indeed follows from \eqref{ext-magri}.

However, we point out that the condition $\ph\in C_{0}^{\infty}[0,1]$ is neither dynamically invariant for the NLS system~\eqref{nls-sys} nor for the mKdV system~\eqref{mkdv-sys}. 
Moreover, in the case of periodic boundary conditions, one verifies using the case that $\ph$ is a nontrivial constant, that the NLS system~\eqref{nls-sys} is not Bi-Hamiltonian in the sense of Magri. Indeed, the identity
\[
  K\partial \Hz_{3} = K_{\star}\partial \Hz_{2}
\]
does not hold for any linear operator $K_{\star}$, since for this choice of the potential $\partial S_{2}$ vanishes, $\partial S_{3}$ does not, and $K$ is invertible. In fact, one infers that~\eqref{ext-magri} generically does not hold for $m$ even. However, one can restrict $K\partial \Hz_{m+1}$ and $K_{2} \partial \Hz_{m}$ to the invariant subspaces
\[
  \Mc_{-} = \setd{\ph(1-x) = -\ph(x)},\qquad
  \Mc_{+} = \setd{\ph(1-x) = \ph(x)},
\]
where all odd, respectively even, derivatives of $\ph$ vanish on the boundary of $[0,1]$. On these spaces~\eqref{ext-magri} holds for $m$ odd -- see Section~\ref{s:fin}.

Furthermore, Theorem~\ref{thm:nls-sym} and Theorem~\ref{thm:freq} are related to \cite{Grebert:2002wb}. In particular, the identities $I_{-n}^{\rest} = I_{n}^{\rest}$ and $\th_{-n}^{\rest} = - \th_{n}^{\rest}$ for any $n\ge 1$ are proved there implying $\om_{-n,2m}^{\rest} = -\om_{n,2m}^{\rest}$.

Finally, we mention the work of  \citet{Zakharov:1986hq}, where multiscale expansions are proposed to discover relations between various integrable PDEs.

\emph{Organization of this paper.} In Section~\ref{s:setup} the \mkdv and \nls action variables as well as the spectral data needed to define them are introduced. In Section~\ref{s:zs-hm-spec} the discriminants of \mkdv and \nls are compared and Theorem~\ref{thm:bnf} (i)-(ii) are proven. In Section~\ref{s:symmetries} the symmetries of the Hamiltonians in the \nls hierarchy under the transformations $\ph\mapsto P\ph$ and $\ph\mapsto R_{\al}\ph$, $\al\in\R$, are obtained and subsequently used in Section~\ref{s:fin} to prove Theorem~\ref{thm:ham-mkdv-nls}, Theorem~\ref{thm:nls-sym}, Theorem~\ref{thm:freq}, and Theorem~\ref{thm:bnf} (iii).

\emph{Acknowledgments.}
We are very grateful to Herbert Koch for valuable discussions and to Thomas Kappeler for his  continued support and helpful feedback on this manuscript.
This work was partially supported by the Swiss Science Foundation.

\section{Setup}
\label{s:setup}

We begin by briefly recalling the definition of the \mkdv and \nls action variables as well as the properties of the spectral data needed to define them -- see e.g. \cite{Kappeler:2003up,Kappeler:2008fk,Grebert:2014iq}.

\textit{\mkdv action variables.} The Miura transform \cite{Miura:1968uq}
\[
  H^{1}_{c} \to H^{0}_{c},\qquad u \mapsto  B(u) = u_{x} + u^{2},
\]
when restricted to $H_{r}^{1}$ where \mkdv is well-posed, maps solution of the defocusing \mkdv equation onto solutions of the \kdv equation. This allows us to use the setup for $\kdv$ as in \cite{Kappeler:2003up} and pull back all defined objects using the Miura transform -- see also \cite{Kappeler:2008fk}. For a \emph{potential} $u\in H_{c}^1$ consider the \emph{Hill operator}
\[
  \Lh(u)=-\partial_{x}^{2}+B(u)
\]
on the interval $[0,2]$ of twice the length of the period of $u$ with periodic boundary conditions. By a slight abuse of notation, the spectrum of $\Lh(u)$ is called the \emph{periodic spectrum of $u$} and is denoted by $\spec(u)$. It is known to be discrete and to consist of a sequence $\mu_{0}^{+}(u),\mu_{1}^{-}(u),\mu_{1}^{+}(u),\dotsc$ of periodic eigenvalues, which, when counted with their multiplicities, can be ordered lexicographically -- first by their real part and second by their imaginary part -- such that
\[
  \mu_0^{+} \lex \mu_1^- \lex \mu_1^+ \lex \dotsb \lex \mu_n^- \lex \mu_n^+ \lex \dotsb,
  \qquad \mu_{n}^{\pm} = n^{2}\pi^{2} + \ell_{n}^{2}.
\]
Here $\ell_{n}^{2}$ denotes a generic $\ell^{2}$-sequence. For any $n\ge 1$ we define the \emph{gap length}
\[
  \dl_n(u) \defl \mu_{n}^+(u) - \mu_{n}^-(u).
\]
When $u$ is real-valued, the periodic spectrum and the gap lengths are real-valued.

To obtain a suitable characterization of the periodic spectrum, let $y_1(x,\mu,u)$ and $y_2(x,\mu,u)$ be the two standard fundamental solutions of $\Lh y = \mu y$ and denote by $\Dh(\mu,u)=y_1(1,\mu,u)+y_2'(1,\mu,u)$ the associated \emph{discriminant}. The periodic spectrum of $u$ is precisely the zero set of the entire function $\Dh^2(\mu) - 4$, and we have the product representation
\begin{equation}
  \label{Dh-sq}
  \Dh^2(\mu) - 4
   = 
  4(\mu_0^{+}-\mu)\prod_{m\ge 1} \frac{(\mu_m^+-\mu)(\mu_m^--\mu)}{m^4\pi^4}.
\end{equation}
The $\mu$-derivative is denoted by $\dDh\defl\partial_{\mu}\Dh$.

For each potential $u\in H_{r}^{1}$ there exists an open neighborhood $\Vh_{u}$ within $H^{1}_{c}$ such that the straight lines
\[
  \Gh_{0} = \setdef{\mu_{0}^{+}-t}{t\ge 0},\qquad 
  \Gh_{n} = [\mu_{n}^{-},\mu_{n}^{+}],\quad n\ge 1,
\]
are disjoint from each other for every potential in $\Vh_{u}$. Actually, for $\Vh_{u}$ sufficiently small, there exist mutually disjoint neighborhoods $(\Uh_{n})_{n\ge 0}\subset \C$, called \emph{isolating neighborhoods}, such that $\Gh_{n}$ is contained in $\Uh_{n}$ for every $n\ge 0$ and every potential in $\Vh_{u}$, and $\Uh_{n} = \setd{\abs{\mu-n^{2}\pi^{2}} \le \pi/4}$ for $n$ sufficiently large. The union of all $\Vh_{u}$ with $u\in H_{r}^{1}$ is denoted by $\Wh$.

To define the action variables in terms of contour integrals in the complex plane, we introduce the \emph{canonical branch} of the square root of $\Dh^{2}-4$ by stipulating on $H_{r}^{1}$ that
\begin{align}
  \label{c-mkdv}
  \ii\sqrt[c]{\Dh^2(\mu)-4}>0 \;\text{ for }\; \mu\in (\mu_0^{+},\mu_1^-).
\end{align}
This root admits an analytic extension onto $(\C\setminus\bigcup_{n\ge 0} \Uh_{n}) \times \Vh_{u}$ for any $u\in \Wh$ -- see also \cite{Kappeler:2003up}.

To proceed, we define for any $u\in \Wh$ on $(\C\setminus\bigcup_{n\ge 0} \Uh_{n}) \times \Vh_{u}$ the mapping
\begin{equation}
  \label{Fh}
    \Fh(\mu)
   = \int_{\mu_{0}^{+}}^{\mu} \frac{\dDh}{\sqrt[c]{\Dh^{2}-4}}\,\dz,
\end{equation}
where the path of integration is chosen to not intersect any open gap except possibly at its endpoints. This mapping is analytic on $(\C\setminus\bigcup_{n\ge 0} \Uh_{n}) \times \Vh_{u}$,
and locally around $\Gh_{n}$
\begin{align*}
  \Fh(\mu) + \ii n\pi
   &= \cosh^{-1}\left(\frac{\Dh(\mu)}{2}\right)\\
   &\defl \log\frac{(-1)^{n}}{2}\left(\Dh(\mu) + \sqrt[c]{\Dh^{2}(\mu)-4}\right),
\end{align*}
with $\log$ denoting the principal branch of the logarithm.
Moreover, for a finite gap potential $u$, the mapping $\Fh$ has the following asymptotic expansion along $\mu = a_{n}^{2}$ with $a_{n} = (n+1/2)\pi$,
\begin{align}
\label{exp-mkdv}
  \Fh\Big|_{\mu = a_{n}^{2}}
   = -\ii a_{n} + \ii\sum_{1\le k\le N} \frac{2 \Hh_{k}}{(2 a_{n})^{2k-1}}
   + O(a_{n}^{-2N-1}),
   \qquad n \to +\infty.
\end{align}
The $n$th \emph{\mkdv action}, $n\ge 1$, of $u\in \Wh$ is then given by
\begin{equation}
  \label{Jn}
  \quad J_{n}
   \defl 
  \frac{-1}{4\pi}\int_{\Gmh_n} \mu^{-1} \Fh(\mu)\,\dmu,
\end{equation}
where $\Gmh_{n}$ denotes any sufficiently close counter clockwise oriented circuit around $\Gh_{n}$ which does not enclose the origin. The action $J_{n}$ vanishes if and only if the gap length $\dl_{n}$ is zero -- see \cite{Kappeler:2008fk,Molnar:2016hq} for details.

\textit{\nls action variables.} 
For a potential $\ph=(\phm,\php)\in \Hc^0_c$, consider the \emph{Zakharov-Shabat operator}
\[
  \Lz(\ph) \defl
  \mat[\bigg]{\,\ii & \\  & -\ii} 
  \frac{\ddd}{\dx} +
  \mat[\bigg]{ & \phm \\ \php & }
\]
on the interval $[0,2]$ with periodic boundary conditions. By a slight abuse of notation, the spectrum of $\Lz(\ph)$ is called the \emph{periodic spectrum of $\ph$} and is denoted by $\spec(\ph)$. It is known to be discrete and to consist of a sequence of pairs of complex eigenvalues $\lm_n^+(\ph)$ and $\lm_n^-(\ph)$, $n\in\Z$, listed with algebraic multiplicities, such that when ordered lexicographically
\[
  \dotsb \lex \lambda_{n-1}^+ \lex \lambda_{n}^- \lex \lambda_{n}^+ \lex \lambda_{n+1}^- \lex \dotsb,\qquad 
  \lm_n^\pm = n\pi + \ell^2_n.
\]
We also define the \emph{gap lengths}
\[
  \gm_{n}(\ph) \defl \lm_{n}^{+}(\ph)-\lm_{n}^{-}(\ph).
\]

Denote by $\Mz(x,\lm,\ph)$ the standard fundamental solution of $\Lz(\ph)\Mz = \lm \Mz$, and introduce the \emph{discriminant} $\Dz(\lm,\ph) \defl  \operatorname{tr} \Mz(1,\lm,\ph)$.
The periodic spectrum of $\ph$ is precisely the zero set of the entire function $\Dz^2(\lm) - 4$, and we have the product representation
\begin{equation}
  \label{Dz-sq}
  \Dz^2(\lm) - 4
   = 
  -4\prod_{n\in\Z}
  \frac{(\lm_n^+-\lm)(\lm_n^--\lm)}{\pi_n^2},
  \qquad
  \pi_n
   \defl 
  \begin{cases}
  n\pi, & n\neq 0,\\
  1, & n=0.
  \end{cases}
\end{equation}
We also need the $\lm$-derivative $\dDz\defl\partial_{\lm}\Dz$.

For any potential $\ph\in \Hc_{r}^{0}$, there exists an open neighborhood $\Vz_{\ph}$ within $\Hc_{c}^{0}$ for which there exist disjoint closed discs $(\Uz_{n})_{n\in\Z}$ centered on the real axis such that $\Gz_{n} \defl [\lm_{n}^{-},\lm_{n}^{+}]$ is contained in the interior of $\Uz_{n}$ for any potential in $\Vz_{\ph}$ and any $n\in\Z$, and $\Uz_{n} = \setd{\abs{\lm-n\pi} \le \pi/4}$ for $\abs{n}$ sufficiently large.
Such discs are called \emph{isolating neighborhoods}, and we denote the union of all $\Vz_{\ph}$ with $\ph\in \Hc_{r}^{0}$ by $\Wz$.

To define the action variables in terms of contour integrals in the complex plane, we introduce the \emph{canonical root} $\sqrt[c]{\Dz^{2}(\lm)-4}$ by stipulating on $\Hc_{r}^{0}$ that
\begin{equation}
  \label{c-nls}
  \ii \sqrt[c]{\Dz^{2}(\lm)-4} > 0,\qquad \lm_{0}^{+} < \lm < \lm_{1}^{-}.
\end{equation}
This root admits an analytic continuation onto $(\C\setminus \bigcup_{n\in\Z} \Uz_{n})\times \Vz_{\ph}$.

To proceed, we define for any $\ph\in \Wz$  on $(\C\setminus\bigcup_{n\in\Z} \Uz_{n}) \times \Vz_{\ph}$ the mapping
\begin{equation}
  \label{Fz}
    \Fz(\lm) \defl
   \frac{1}{2}\left(\int_{\lm_{0}^{-}}^{\lm} \frac{\dDz(z)}{\sqrt[c]{\Dz^{2}(z)-4}}\,\dz + \int_{\lm_{0}^{+}}^{\lm} \frac{\dDz(z)}{\sqrt[c]{\Dz^{2}(z)-4}}\,\dz\right).
\end{equation}
This map is analytic on $(\C\setminus\bigcup_{n\in\Z} \Uz_{n}) \times \Vz_{\ph}$ with gradient
\begin{equation}
\label{grad-F}
  \gradz \Fz = \frac{\gradz \Dz}{\sqrt[c]{\Dz^{2}-4}}.
\end{equation}
Furthermore, $F(\lm_{0}^{+}) = F(\lm_{0}^{-}) = 0$, and
\[
  F(\lm) = \int_{\lm_{0}^{-}}^{\lm} \frac{\dDz(z)}{\sqrt[c]{\Dz^{2}(z)-4}}\,\dz
         = \int_{\lm_{0}^{+}}^{\lm} \frac{\dDz(z)}{\sqrt[c]{\Dz^{2}(z)-4}}\,\dz.
\]
If $\ph\in\Hc_{r}^{0}$, then locally around $\Gz_{n}$
\[
  \Fz(\lm) + \ii n\pi = \cosh^{-1}\left(\frac{\Dz(\lm)}{2}\right)
  = \log\frac{(-1)^{n}}{2}\left(\Dz(\lm) + \sqrt[c]{\Dz^{2}(\lm)-4}\right).
\]
Moreover, if $\ph$ is a finite gap potential, then there exists $\Lm > 0$ such that
\begin{align}
  \label{exp-nls}
  \Fz(\lm,\ph)
   = -\ii \lm + \ii\sum_{n\ge 1} \frac{\Hz_{n}(\ph)}{(2\lm)^{n}},
   \qquad \abs{\lm} > \Lm.
\end{align}
For $\ph\in W$ the $n$th \emph{\nls action variable}, $n\in\Z$, is given by
\begin{equation}
  \label{In}
  I_n = -\frac{1}{\pi}\int_{\Gmz_n} \Fz(\lm)\,\dlm,
\end{equation}
with $\Gmz_{n}$ being a sufficiently close counter clockwise oriented circuit around $\Gz_{n}$.
The action $I_{n}$ vanishes if and only if the gap length $\gm_{n}$ is zero
 -- see \cite{Grebert:2014iq,Molnar:2014vg} for details.

\section{Identity for the discriminants}
\label{s:zs-hm-spec}

In this section we establish the following identity relating the discriminants of the Zakharov-Shabat operator with the one of a corresponding Hill operator and discuss several applications.
 
\begin{thm}
\label{comp-dl}
For all $\lm\in\C$ and $u\in H_{c}^{1}$,
\[
  \Dh(\lm^{2},u) = \Dz(\lm,\ph_{u}),
\]
where $\ph_{u} \defl (u,u)$.\fish
\end{thm}

The proof of this theorem is based on an observation by Chodos \cite{Chodos:1980du} relating the fundamental solutions of the Hill operator $\Lh(u)$ and the Zakharov-Shabat operator $\Lz(\ph_{u})$.
Given $u\in H_{c}^{1}$ and $\lm\in \C$, define
\[
  A(x,\lm,u) \defl \mat[\bigg]{1 & \ii\\ u-\ii\lm & \ii u - \lm }.
\]
Note that $\det A(x,\lm,u) = -2\lm$, hence $A(x,\lm,u)$ is invertible for any $\lm\neq 0$, $0\le x \le 1$, and $u\in H_{c}^{1}$.
Furthermore, $A(x,\lm,u)$ is $1$-periodic in $x$. Finally, denote by $\Mh$ the fundamental solution of $\Lh$ and by $\Mz$ the one of $\Lz$.

\begin{lem}
\label{comp-fs}
Suppose $u\in H^{1}_{c}$. If for some $f=(f_1,f_2)\in \Hc^{2}_{c}([0,1])$ and $\lm\in \C$,
\[
  \Lz(\ph_{u})f = \lm f,
\]
then
\[
  \Lh(u)(f_{1}+\ii f_{2}) = \lm^{2}(f_{1}+\ii f_{2}).
\]
Moreover, if $\lm\neq 0$, then
\[
 \Mh(x,\lm^{2},u) = 
  A(x,\lm,u)\Mz(x,\lm,\ph_{u})A(0,\lm,u)^{-1},
\]
In particular, at $x=1$,
\[
  \Mh(1,\lm^{2},u) = 
  A(0,\lm,u)\Mz(1,\lm,\ph_{u})A(0,\lm,u)^{-1}.\fish
\]
\end{lem}

\begin{proof}
If $f$ is a solution of $\Lz(\ph_{u}) f = \lm f$, then
\[
  \partial_{x} f = \mat[\bigg]{-\ii \lm & \ii u\\ -\ii u & \ii \lm}f,
  \qquad
  \partial_{x}^{2} f = \ii (\partial_{x}u) J f + (-\lm^{2} + u^{2})f,
  \quad
  J = \mat[\bigg]{ & 1\\ -1 & },
\]
and hence
\[
  (-\partial_{x}^{2} + u^{2} + \partial_{x}u)(f_{1}+\ii f_{2}) = \lm^{2}(f_{1}+\ii f_{2}).
\]
Let $\Mz = \smat{m_{1} & m_{2}\\ m_{3} & m_{4}}$ denote the entries of the fundamental solution of $\Lz$, and define
\[
  y_{1} = m_{1} + \ii m_{3},\qquad y_{2} = m_{2} + \ii m_{4},
\]
then by the preceding calculation
\[
  y_{1}' = (u-\ii \lm)m_{1} + (\ii u - \lm) m_{3},\qquad
  y_{2}' = (u-\ii \lm)m_{2} + (\ii u - \lm) m_{4}.
\]
Thus $Y = \left(\begin{smallmatrix}y_{1} & y_{2}\\y_{1}' & y_{2}'\end{smallmatrix}\right)$ is a fundamental solution of $-y'' + B(u)y = \lm y$ with
\[
  Y = A\Mz = \begin{pmatrix}
  1 & \ii \\
  u- \ii\lm & \ii u - \lm
  \end{pmatrix}
  \begin{pmatrix}
  m_{1} & m_{2}\\
  m_{3} & m_{4}
  \end{pmatrix}.
\]
As $A(0,\lm,u)$ is invertible if and only if $\lm\neq 0$, and $u$ is 1-periodic in $x$, the claim follows.\qed
\end{proof}

\begin{proof}[Proof of Theorem~\ref{comp-dl}.]
By Lemma~\ref{comp-fs}, the fundamental solutions ${\Mh}(1,\lm^{2},u)$ and ${\Mz}(1,\lm,\ph_{u})$, evaluated at $x=1$,  are conjugated for $\lm\in \C\setminus\setd{0}$. Thus their discriminants coincide, $\Dh(\lm^{2},u) = \Dz(\lm,\ph_{u})$. By continuity this identity also holds for $\lm = 0$.\qed
\end{proof}

Consequently, as already noted in \cite{Grebert:2002wb}, the discriminant $\Dz(\lm,\ph_{u})$ for $u\in H^{1}_{c}$ is an even function of $\lm$. Recall from~\eqref{Dz-sq} that the periodic spectrum of $\ph_{u}$ is precisely the zero set of $\Dz^{2}(\lm,\ph_{u})-4$. Thus, it follows from the asymptotic behavior $\lm_{n}^{\pm} = n\pi + \ell_{n}^{2}$ and the lexicographical ordering that
\begin{align}
  \label{eq:spectralsymmetrie}
  \lm_{-n}^{\mp}(\ph_{u}) = -\lm_{n}^{\pm}(\ph_{u}),\qquad n\ge 0.
\end{align}
Further symmetries of the discriminant will be obtained in Section~\ref{s:symmetries}.

\begin{lem}
\label{comp-spec}
For every $u\in H_{r}^{1}$,
\[
  \mu_{0}^{+}(u) = (\lm_{0}^{+}(\ph_{u}))^{2} = (\lm_{0}^{-}(\ph_{u}))^{2},\qquad
  \mu_{n}^{\pm}(u) = (\lm_{n}^{\pm}(\ph_{u}))^{2},\quad n\ge 1,
\]
and $\mu_{n}^{\pm}(u)$ has the same geometric multiplicity as $\lm_{n}^{\pm}(\ph_{u})$.
In particular, $\Gz_{n}^{2}(\ph_{u}) = \Gh_{n}(u)$ for any $n\ge 1$, and $\dl_{n}(u) = 0$ iff $\gm_{n}(\ph_{u}) = 0$.\fish
\end{lem}

\begin{proof}
By Theorem~\ref{comp-dl}, $\mu = \lm^{2}$ is an eigenvalue of $\Lh(u)$ if and only if $\lm$ is an eigenvalue of $\Lz(\ph_{u})$. If $u$ is real-valued, then the periodic spectra of $u$ and $\ph_{u}$ are real, and due to the symmetry and the lexicographical ordering
\[
  0\le \mu_{0}^{+} \le \mu_{1}^{-} \le \mu_{1}^{+} \le \dotsb,\qquad
  \dotsb \le \lm_{0}^{-} \le 0 \le \lm_{0}^{+}\le \lm_{1}^{-} \le \lm_{1}^{+} \le \dotsb.
\]
Consequently, $\mu_{0}^{+} = (\lm_{0}^{+})^{2} = (\lm_{0}^{-})^{2}$, and $\mu_{n}^{-} = (\lm_{n}^{-})^{2}$ as well as $\mu_{n}^{+} = (\lm_{n}^{+})^{2}$ for any $n\ge 1$. Thus $\Gz_{n}^{2} = [(\lm_{n}^{-})^{2},(\lm_{n}^{+})^{2}] = \Gh_{n}$ for any $n\ge 1$.

Finally, $\mu_{n}^{+} = \mu_{n}^{-}$ if and only if $\lm_{n}^{+} = \lm_{n}^{-}$, and the fundamental solutions $\Mh(1,\mu_{n}^{\pm},u)$ and $\Mz(1,\lm_{n}^{\pm},\ph_{u})$ are conjugated by Lemma~\ref{comp-fs}. Thus there exist two linear independent eigenfunctions for $\mu_{n} = \mu_{n}^{+} = \mu_{n}^{-}$ if and only if they exist for $\lm_{n} = \lm_{n}^{+} = \lm_{n}^{-}$.\qed
\end{proof}

\begin{proof}[Proof of Theorem~\ref{thm:bnf} (i).]
For $u\in H_{r}^{1}$ it follows from \eqref{Jn}, \eqref{In}, and Lemma~\ref{comp-spec} that for any $n\ge 1$
\[
  J_{n}(u) = 0
   \;\iff\;
  \dl_{n}(u) = 0
   \;\iff\;
  \gm_{n}(\ph_{u}) = 0
   \;\iff\;
  I_{n}(\ph_{u}) = 0.
\]
One concludes from \eqref{Dz-sq} -- see also \cite{Grebert:2014iq} -- that $\Dz(\lm,\ph)-2$ vanishes at $\lm_{0}^{\pm}$ and is strictly positive on the interior of $\Gz_{0}$. 
By \eqref{eq:spectralsymmetrie}, $\Gz_{0}(\ph_{u})$ is a symmetric interval around zero, hence $I_{0}(\ph_u)$ is strictly positive if any only if
\[
  \Dz(0,\ph_u) = \Dh(0,u) \neq 2.
\]
It was observed in \cite{Kappeler:2008fk} that
\begin{align}
  \label{av-cosh}
  \Dh(0,u) = 2\cosh([u]),
\end{align}
hence $I_{0}(\ph_u)$ vanishes if and only if the average $[u]$ vanishes.\qed
\end{proof}

\begin{lem}
\label{comp-croot}
For $u\in H_{r}^{1}$ and $\lm\in \C\setminus\bigcup_{n\in\Z} \Gz_{n}$ with $\Re \lm > 0$,
\[
  \sqrt[c]{\Dh^2(\lm^2,u)-4} = \sqrt[c]{\Dz^2(\lm,\ph_u)-4}.\fish
\]
\end{lem}

\begin{proof}
On $D = (\C\setminus \bigcup_{n\in\Z} \Gz_{n}) \cap \setd{\Re \lm > 0}$ the canonical \nls root of $\ph_{u}$ is analytic, and, by Lemma~\ref{comp-spec}, $D$ is mapped by $\lm\mapsto \lm^{2}$ onto the domain $\C\setminus \bigcup_{n\ge 0} \Gh_{n}$ where the canonical \mkdv root of $u$ is analytic. Moreover, $\Dh^{2}(\lm^{2},u) - 4 = \Dz^{2}(\lm,\ph_u)-4$, by Theorem~\ref{comp-dl}, hence these roots differ at most by a sign. Recall from~\eqref{c-mkdv} that the canonical \mkdv root is chosen such that
\[
  \ii \sqrt[c]{\Dh^{2}(\mu,u)-4} > 0,\qquad \mu_{0}^{+} < \mu < \mu_{1}^{-},
\]
and from~\eqref{c-nls} that the canonical \nls root is chosen such that
\[
  \ii \sqrt[c]{\Dz^{2}(\lm,\ph_u)-4} > 0,\qquad \lm_{0}^{+} < \lm < \lm_{1}^{-}.
\]
Since $\mu_{0}^{+} = (\lm_{0}^{+})^{2}$ and $\mu_{1}^{-} = (\lm_{1}^{-})^{2}$, both roots have the same sign provided that $\Re \lm >0$.\qed
\end{proof}

\begin{lem}
\label{comp-F}
Suppose $u\in H_{r}^{1}$, then on $\C\setminus\bigcup_{n\in\Z} \Gz_{n}$ provided $\Re \lm > 0$,
\[
  \Fh(\lm^{2},u) = \Fz(\lm,\ph_u).\fish
\]
\end{lem}

\begin{proof}
Since $\dDz(\lm,\ph_u) = 2\lm \dDh(\lm^{2},u)$ by Theorem~\ref{comp-dl}, we conclude with \eqref{Fz} and Lemma~\ref{comp-croot},
\begin{align*}
  \Fz(\lm,\ph_{u})
    &= \int_{\lm_{0}^{+}}^{\lm}
       \frac{\dDz(z,\ph_u)}{\sqrt[c]{\Dz^{2}(z,\ph_u)-4}}\,\dz
    = \int_{\lm_{0}^{+}}^{\lm}
       \frac{\dDh(z^{2},u)}{\sqrt[c]{\Dh^{2}(z^{2},u)-4}}\, 2z\,\dz.
\end{align*}
Now substituting $w = z^{2}$, and using that $\mu_{0}^{+} = (\lm_{0}^{+})^{2} \ge 0$ and $\Re\lm > 0$, yields in view of \eqref{Fh}
\begin{align*}
  \Fz(\lm,\ph_{u})
    &= \int_{(\lm_{0}^{+})^{2}}^{\lm^{2}}
       \frac{\dDh(w,u)}{\sqrt[c]{\Dh^{2}(w,u)-4}}\, \dw
     = \Fh(\lm^{2},u).\qed
\end{align*}
\end{proof}

At this point, we may recover Chodos' observation in the framework of the \nls hierarchy, which in addition allows us to prove that the Hamiltonians $\Hz_{2m}$ for any $m\ge 1$ vanish at any point $\ph_{u}$. To simplify notation, for any functional $f$ we set $f^{\rest}(u) \defl f(\ph_{u})$.

\begin{prop}
\label{ham-id-mkdv-nls}
The Hamiltonians of the \mkdv and \nls hierarchies satisfy for every $m\ge 1$ 
\[
  \Hh_{m} = \frac{1}{2}\Hz_{2m-1}^{\rest}\text{ on }H_{c}^{m-1},\qquad
  \Hz_{2m}^{\rest} = 0\text{ on }H_{c}^{m}.\fish
\]
\end{prop}

\begin{proof}
It immediately follows from the preceding lemma and the expansions \eqref{exp-mkdv} and \eqref{exp-nls} that for each $m\ge 1$,
\[
  \Hh_{m}(u)  = \frac{1}{2}\Hz_{2m-1}(\ph_u) \quad\text{on}\quad H_{r}^{m-1},\qquad
  \Hz_{2m}(\ph_{u}) = 0 \quad\text{on}\quad H_{r}^{m}.
\]
Since both hand sides are analytic in $u$, these identities extend to all of $H_{c}^{m-1}$ and $H_{c}^{m}$, respectively, by Lemma~\ref{ana-vanish}.\qed
\end{proof}

To be able to compare the actions of \mkdv and \dnls, it is convenient to introduce actions defined  on integer levels $k$ -- see \cite{McKean:1997ka,Grebert:2014iq}. More precisely, for $u\in \Wh$ the $n$th \mkdv action, $n\ge 1$, on level $k\in \Z$ is defined by
\begin{equation}
  \label{Jnk}
    J_{n,k}(u)
   \defl
  -\frac{1}{4\pi}\int_{\Gmh_n} \mu^{k-2} \Fh(\mu,u)\,\dmu,
\end{equation}
where $\Gmh_{n}$ is a sufficiently close circuit around $\Gh_{n}$ which does not enclose the origin.
Note that $J_{n} = J_{n,1}$.
Similarly, for $\ph_{u}\in \Wz$, the $n$th \nls action, $n\in\Z$, on level $k\in \Z$ is defined by
\begin{equation}
  \label{Ink}
  I_{n,k}(\ph_{u}) \defl -\frac {1}{\pi}\int_{\Gmz_n}\lm^{k-1} \Fz(\lm,\ph_{u})\,\dlm,
\end{equation}
where $\Gmz_{n}$ is a sufficiently close circuit around $\Gz_{n}$ which in the case $n\neq 0$ does not enclose the origin, while $\Gmz_{0}$ is a circuit around the origin.
Note that $I_{n} = I_{n,1}$.

\begin{lem}
\label{comp-In-Jn}
On  $H_{r}^{1}$ for any $n\ge 1$ and any $k\in\Z$,
\[
  2J_{n,k}(u) = -\frac{1}{\pi}\int_{\Gmz_{n}} \lm^{2k-3} \Fz(\lm,\ph_u)\,\dlm = I_{n,2k-2}(\ph_u).
\]
In particular, $I_{n,2k-2}$ is an analytic extension of $(u,u)\mapsto 2J_{n,k}(u)$ onto an open neighborhood of $\Ec_{r}^{0}$ within $\Hc_{c}^{0}$.\fish
\end{lem}

\begin{proof}
For any $n\ge 1$, the map $\lm\mapsto \lm^{2}$ maps any sufficiently close circuit $\Gmz_{n}$ around $\Gz_{n}$ bijectively onto a circuit $\Gmh_{n}$ around $\Gh_{n}$. Consequently, by the transformation formula and the previous lemma
\begin{align*}
  I_{n,2k-2}^{\rest} &= - \frac{1}{\pi}\int_{\Gmz_{n}} \lm^{2k-3} \Fz^{\rest}(\lm)\,\dlm\\
  &= - \frac{1}{2\pi}\int_{\Gmz_{n}} \lm^{2k-4} \Fh(\lm^{2})\,2\lm\,\dlm\\
  &= - \frac{1}{2\pi}\int_{\Gmh_{n}} \mu^{k-2} \Fh(\mu)\,\dmu = 2J_{n,k}.\qed
\end{align*}
\end{proof}

\begin{proof}[Proof of Theorem~\ref{thm:bnf} (ii).]
After possibly shrinking $\Wh$, we may assume that $u\in \Wh$ implies $\ph_{u}\in \Wz$. As a result, $I_{n}^{\rest}(u) = I_{n}(\ph_{u})$, $n\in\Z$, defines an analytic function on $\Wh$. Moreover, $\gradh_{u} \Dz^{\rest}(\lm) = \gradh_{u} \Dh(\lm^{2})$ by Theorem~\ref{comp-dl}, hence
\[
  \gradh I_{n}^{\rest}(u)
  = -\frac{1}{\pi}\int_{\Gm_{n}}\frac{\gradh_{u}\Dz^{\rest}(\lm)}{\sqrt[c]{\Dz^{2}(\lm,\ph_{u})-4}}\,\dlm
  = -\frac{1}{\pi}\int_{\Gm_{n}}\frac{\gradh_{u}\Dh(\lm^{2})}{\sqrt[c]{\Dh^{2}(\lm^{2},u)-4}}\,\dlm.
\]
As shown in \cite[Lemma 10.2]{Kappeler:2003up}, we have $\gbr{\Dh(\mu_{1}),\Dh(\mu_{2})} = 0$ for any $\mu_{1},\mu_{2}\in\C$. Hence, for any $n\in\Z$ and $m\ge 1$ on $\Wh$,
\[
  \gbr{I_{n}^{\rest},J_{m}} 
  = \frac{1}{4\pi^{2}}\int_{\Gmz_{n}}\int_{\Gmh_{m}}\frac{1}{\mu}
    \frac{\gbr{\Dh(\lm^{2}),\Dh(\mu)}}{\sqrt[c]{\Dh^{2}(\lm^{2})-4}\sqrt[c]{\Dh^{2}(\mu)-4}}
    \,\dlm\,\dmu = 0.
\]
Consequently, each $I^{\rest}_{n}\circ\Oh^{-1}$, where $\Oh$ denotes the \mkdv Birkhoff map, is a real analytic function of the actions $(J_{m})_{m\ge 1}$ and the average $[u]$ alone. 

Conversely, by the preceding lemma each $J_{m}$, $m\ge 1$, extends to an analytic function $\tilde J_{m} = I_{m,0}$ on an open neighborhood of $\Ec_{r}^{0}$ within $\Hc_{c}^{0}$. Moreover,
\[
  \nbr{\tilde J_{m},I_{n}}
   = \nbr{I_{m,0},I_{n}}
   = \frac{1}{\pi^{2}}\int_{\Gmz_{n}}\int_{\Gmz_{m}}
     \frac{1}{z}\frac{\nbr{\Dz(z),\Dz(w)}}{\sqrt[c]{\Dz^{2}(z)-4}\sqrt[c]{\Dz^{2}(w)-4}}\,\dz\,\dw = 0,
\]
using that $\nbr{\Dz(z),\Dz(w)} = 0$ for any $z,w\in \C$ by \cite[Lemma 8.3]{Grebert:2014iq}. So, with $\Oz$ denoting the \nls Birkhoff mapping, $\tilde{J}_{m}\circ\Oz^{-1}$ is a real analytic function of  the actions $(I_{n})_{n\in\Z}$ alone. For the average we have for $u\in H_{r}^{1}$ by \eqref{av-cosh}
\[
  [u]
   = \cosh^{-1}\frac{\Dz(0,\ph_u)}{2}
   = -\int_{\lm_{0}^{-}}^{\tau_{0}} \frac{\dDz(\lm,\ph_u)}{\sqrt[c]{\Dz^{2}(\lm,\ph_u)-4}}\,\dlm,
\]
where $\tau_{0} = (\lm_{0}^{+}+\lm_{0}^{-})/2=0$ and the path of integration is chosen to run on the right hand side of the straight line connecting $\lm_{0}^{-}$ and $\tau_{0}$ in the complex plane. One shows by exactly the same arguments as in the proof of \cite[Proposition A2]{Molnar:2014vg} that the latter defines a real analytic function on all of $\Hc_{r}^{0}$. Since it only depends on the periodic spectrum, it is a real analytic function of the actions alone.\qed
\end{proof}

\section{Symmetries of the Zakharov-Shabat discriminant}
\label{s:symmetries}

In this section we obtain several symmetries of the Zakharov-Shabat discriminant under the transformations
\[
  \ph \mapsto P\ph,\quad \ph\mapsto R_{\al}\ph,\quad\ph\mapsto T\ph\qquad
  P \defl \mat[\bigg]{
  0 & 1\\
  1 & 0},
  \quad
  R_{\al} \defl \mat[\bigg]{
  \e^{\ii\al} & 0\\
  0   & \e^{-\ii \al}},
  \quad \al\in\R,
\]
and $T\ph(x) = \ph(1-x)$. By Theorem \ref{comp-dl} those symmetries translate into corresponding symmetries of the Hill discriminant.
To simplify notation, let
\[
  J \defl \mat[\bigg]{
  0 & 1\\
  -1 & 0},
  \qquad
  R \defl R_{\pi/2} = 
  \mat[\bigg]{
  \ii & 0\\
  0   & -\ii}.
\]

\begin{thm}
\label{thm:sym-dl}
The discriminant $\Dz$ has the following symmetries
\begin{equivenum}

\item 
$\Dz(\lm,\ph) = \Dz(-\lm,\cph) = \Dz(-\lm,T\ph) = \Dz(\lm,R_{\al}\ph)$ for all $\lm\in\C$;

In particular, for all $n\in\Z$,
\begin{align*}
  &\lm_{n}^{\pm}(\ph) = -\lm_{-n}^{\mp}(\cph) = -\lm_{-n}^{\mp}(T\ph) = \lm_{n}^{\pm}(R_{\al}\ph),\\
  &G_{n}(\ph) =  -G_{-n}(\cph) = -G_{-n}(T\ph) = G_{n}(R_{\al}\ph),
\end{align*}
hence one can choose $\Wz$ to be invariant under $P$, $T$, and $R_{\al}$ for any $\al\in\R$.

\item 
$
      \dDz(\lm,\ph)
   = -\dDz(-\lm,\cph)
   = -\dDz(-\lm,T\ph) 
   =  \dDz(\lm,R_{\al}\ph)
$ for all $\lm\in\C$ and all $\al\in\R$.

\item
$\gradz \Dz(\lm,\ph) = P\gradz \Dz(-\lm,\cph) = T\gradz \Dz(-\lm,T\ph) = R_{\al}\gradz \Dz(\lm,R_{\al}\ph)$ for all $\lm\in\C$ and $\al\in\R$.

\item
If $\ph\in\Wz$, then for all $\lm\in\C\setminus\bigcup_{n\in\Z} \Gz_{n}$ and all $\al\in\R$,
\[
  \sqrt[c]{\Dz^{2}(\lm,\ph)-4}
   = -\sqrt[c]{\Dz^{2}(-\lm,\cph)-4}
   = -\sqrt[c]{\Dz^{2}(-\lm,T\ph)-4}
   = \sqrt[c]{\Dz^{2}(\lm,R_{\al}\ph)-4}.
\]


%
\item
$\partial_{x}\partial \Dl(\lm,\ph) - 2\lm R\partial \Dl(\lm,\ph) = \xi(\lm,\ph) J\ph$ where the function
\[
  \xi(x,\lm,\ph) \defl \p*{(\grave{m}_{1}(\lm,\ph) - \grave{m}_{4}(\lm,\ph))
   +
   2\ii \int_{0}^{x}(R\ph \cdot \partial \Dl(\lm,\ph))\,\dy}
\]
is $1$-periodic in $x$ and satisfies $\xi(x,\lm,\ph) = -\xi(x,-\lm,P\ph) = -\xi(1-x,-\lm,T\ph) = \xi(x,\lm,R_{\al}\ph)$ for all $x\in \R$, $\lm\in\C$, and $\al\in\R$.\fish
\end{equivenum}

\end{thm}

\emph{Fundamental solution.}
Since $\Dl(\lm,\ph)$ is the trace of the fundamental solution $\grave{M}(\lm,\ph) = M(x,\lm,\ph)\big|_{x=1}$, it follows from Lemma~\ref{F-sol-sym} that $\Dl(\lm,\ph) = \Dl(-\lm,\cph)$, $\Dl(\lm,R_{\al}\ph) = \Dl(\lm,\ph)$, and $\Dl(-\lm,T\ph) = \Dl(\lm,\ph)$.
Differentiating these identities with respect to $\lm$ and $\ph$ gives items (ii) and (iii) of Theorem~\ref{thm:sym-dl}.

Recalling from \eqref{Dz-sq} that the periodic spectrum is the zero set of $\Dz^{2}-4$, we conclude that
\[
  \lm \in \spec(\ph) \iff -\lm\in \spec(\cph) \iff -\lm\in \spec(T\ph) \iff \lm \in \spec(R_{\al}\ph).
\]
From the lexicographical ordering and the asymptotic behavior $\lm_{n}^{\pm} = n\pi + \ell_{n}^{2}$ we further infer that $\lm_{n}^{\pm}(\ph) = -\lm_{-n}^{\mp}(\cph) = -\lm_{-n}^{\mp}(T\ph) = \lm_{n}^{\pm}(R_{\al}\ph)$ for any $n\in\Z$.
This proves item (i) of Theorem~\ref{thm:sym-dl}.

\emph{Canonical root}. Clearly, $\Dz^{2}(\lm,\ph)-4 = \Dz^{2}(-\lm,\cph)-4 = \Dz^{2}(-\lm,T\ph)-4 = \Dz^{2}(\lm,R_{\al}\ph)-4$ on $\Wz$, hence
\[
  \sqrt[c]{\Dz^2(\lm,\ph)-4}
   = \vs_{P}\sqrt[c]{\Dz^2(-\lm,\cph)-4}
   = \vs_{T}\sqrt[c]{\Dz^2(-\lm,\cph)-4}
   = \vs_{R_{\al}}\sqrt[c]{\Dz^2(\lm,R_{\al}\ph)-4},
\]
where the signs $\vs_{P}$, $\vs_{T}$, and $\vs_{R_{\al}}$ have modulus one, are locally constant in $\ph$, and independent of $\lm$ as $\Dz^{2}(\lm)-4$ does not vanish on $\C\setminus\bigcup_{\gm_{n}\neq 0} \Gz_{n}$. The straight line connecting $\ph$ and the origin is compact in $W$, and further $\sqrt[c]{\Dz^2(\lm)-4}\big|_{\ph=0}=-2\ii\sin(\lm)$, hence $\vs_{P} = \vs_{T} \equiv -1$ and $\vs_{R_{\al}} = 1$. This proves (iv) of Theorem~\ref{thm:sym-dl}.

\emph{Gradient Symmetry.}
The gradient of $\Dz$ can be represented by the components of $M$ -- see \cite[Section 4]{Grebert:2014iq}. To further simplify notation we denote $\grave{M} \defl M\big|_{x=1}$. Let $\zeta_\pm$ denote the eigenvalues of $\grave{M}$, then for all $\lm$ with $\grave{m_2}\neq0$,
\begin{align}
  \label{grad-dl-m2}
  \ii\gradz\Dz = \grave{m_2}f_+\star f_-,\qquad
  f_\pm \defl M\mat[\bigg]{
    1 \\ 
    \frac{\zeta_\pm-\grave{m_1}}{\grave{m_2}}
    }
\end{align}
where $ \smat{h_1\\h_2} \star \smat{k_1\\k_2}=\smat{h_2k_2\\h_1k_1}$.
The function $\lm \mapsto \grave{m_2}(\lm,\ph)$ vanishes identically if and only if $\ph$ is the zero potential, thus for $\ph\neq 0$ we have $\grave{m_2}(\lm,\ph)\neq 0$ for generic $\lm$.
Denote $\Phi = \smat{ & \ph_{1} \\ \ph_{2} & }$, then we can write
$\partial_{x}f_{\pm} = R(\Phi-\lm)f_{\pm}$ as $f_{\pm}$ is a solution of $Lf_{\pm} = \lm f_{\pm}$. A straightforward computation shows $(Ra)\star b = a\star(Rb) = -R(a\star b)$ and 
\[
  (R\Phi a) \star b + a \star (R\Phi b) = \ii P\Phi((Ja)\star b + a\star (Jb))
\]
for any two vectors $a,b$. Consequently,
\begin{align*}
  \partial_{x}(f_{+}\star f_{-})
  &=(\partial_{x} f_{+}\star f_{-}) + (f_{+}\star \partial_{x} f_{-})
  \\&= (R(\Phi-\lm)f_{+}\star f_{-}) + (f_{+}\star R(\Phi-\lm)f_{-})
  \\&= 2\lm R(f_{+}\star f_{-}) + \ii P\Phi\bigl((J f_{+}\star f_{-}) + (f_{+}\star J f_{-})\bigr).
\end{align*}
We conclude from \eqref{grad-dl-m2} that for generic $\lm$
\[
  \partial_{x}\gradz\Dz(\lm,\ph) - 2\lm R \gradz \Dz(\lm,\ph) =
    P\Phi,\qquad \Pi = \grave{m_{2}}\bigl((J f_{+}\star f_{-}) + (f_{+}\star J f_{-})\bigr).
\]
Note that the function $\Pi$ is one periodic since $f_{\pm}(1) = \zt_{\pm}f_{\pm}(0)$ and $\zt_{+}\zt_{-} = 1$. We proceed by computing the $x$-derivative of $\Pi$. To this end, we compute
\begin{align*}
  \partial_{x}(Jf_{\pm}\star f_{\mp}) 
  &= (JR(\Phi-\lm)f_{\pm}\star f_{\mp}) + (Jf_{\pm}\star R(\Phi-\lm)f_{\mp})\\
  &= (JR\Phi f_{\pm}\star f_{\mp}) + (Jf_{\pm}\star R\Phi f_{\mp})
      - \lm \p*{ (JR f_{\pm}\star f_{\mp}) + (Jf_{\pm}\star R f_{\mp}) }\\
  &= (JR\Phi f_{\pm}\star f_{\mp}) + (Jf_{\pm}\star R\Phi f_{\mp}),
\end{align*}
where we used that $JRa\star b + Ja\star Rb = 0$.
Furthermore, note that $(JR\Phi a\star b) + (J a\star R\Phi b) = (JR\Phi b\star a) + (J b\star R\Phi a) = -((R\ph) \cdot (a\star b))e_{0}$ where $a\cdot b = a_{1}b_{1}+a_{2}b_{2}$ and $e_{0} = (1,-1)$. As a consequence,
\[
  \partial_{x}(Jf_{+}\star f_{-}  + f_{+}\star Jf_{-}) = -2((R\ph) \cdot (f_{+}\star f_{-}))e_{0},
\]
so that
\[
  \partial_{x}\Pi = 2 (P\Phi P - \Phi P) \partial \Dl
                = 2\ii (R\ph\cdot \partial \Dl)e_{0},\qquad
                e_{0} = \begin{pmatrix}1\\-1
                \end{pmatrix}.
\]
Here $a\cdot b \defl a_{1}b_{1} + a_{2}b_{2}$. Since $\Pi(0) = (\grave{m}_{1} - \grave{m}_{4})e_{0}$ we conclude
\[
  \Pi = \p*{(\grave{m}_{1} - \grave{m}_{4}) + 2\ii \int_{0}^{x}\p*{R\ph\cdot \partial \Dl}\,\dy} e_{0}.
\]
Note that $P\Phi e_{0} = J\ph$ so finally
\[
  \partial_{x}\partial \Dl - 2\lm R\partial \Dl
   = 
   \p*{(\grave{m}_{1} - \grave{m}_{4})
   +
   2\ii \int_{0}^{x}(R\ph \cdot \partial \Dl)\,\dy}J\ph.
\]
The properties of $\xi$ follow immediately from the properties of $\Dl$.
This completes the proof of Theorem~\ref{thm:sym-dl}.

\begin{cor}
\label{sym-F}
On $W$ for any $\lm\in \C\setminus\bigcup_{\gm_{n}\neq 0} \Gz_{n}(\ph)$,
\begin{align*}
  F(\lm,\ph)        &= \phantom{P\gradz\!\!\!} -F(-\lm,\cph)
                     = \phantom{T\gradz\!\!} -F(-\lm,T\ph)
                     = \phantom{R_{\al}\gradz}F(\lm,R_{\al}\ph) 
                     ,\\
  \gradz F(\lm,\ph) &= -P \gradz F(-\lm,P\ph)
                     = -T \gradz F(-\lm, T\ph)
                     = R_{\al}\gradz F(\lm,R_{\al}\ph)          
                     .\fish
\end{align*}
\end{cor}

\begin{proof}
Suppose $\ph\in \Wz$. In view of Theorem~\ref{thm:sym-dl},
\begin{align*}
  F(-\lm,\cph)
   &= \phantom{-}\int_{\lm_{0}^{-}(\cph)}^{-\lm} \frac{\dDz(z,\cph)}{\sqrt[c]{\Dz^{2}(z,\cph)-4}}\,\dz\\
   &= -\int_{\lm_{0}^{+}(\ph)}^{\lm} \frac{\dDz(z,\ph)}{\sqrt[c]{\Dz^{2}(z,\ph)-4}}\,\dz
          = -F(\lm,\ph).
\end{align*}
In a similar way, one verifies that $F(\lm,R_{\al}\ph) = -F(-\lm,T\ph) = F(\lm,\ph)$. The identity for the gradients follows by differentiation.\qed
\end{proof}

\begin{cor}
\label{sym-grad-sym-F}
\begin{equivenum}
\item
If $\cph = R_{\al}\ph$ for some $\al\in\R$, then for all $\lm\in\C$
\[
  \partial_{x}\partial F(\lm,\ph) - 2\lm R \partial F(\lm,\ph) = 
  \partial_{x}\partial F(-\lm,\ph) + 2\lm R \partial F(-\lm,\ph).
\]
\item
If $T\ph = \pm \ph$, then for all $\lm\in\C$,
\begin{align*}
  &\partial_{x}\gradz F(\lm,\ph) - 2\lm R \partial F(\lm,\ph)
  - 2\ii \int_{0}^{x} (R\ph\cdot \partial F(\lm,\ph))\,\dy\\
   &\quad = 
  \partial_{x}\gradz F(-\lm,\ph) + 2\lm R \partial F(-\lm,\ph)
  - 2\ii \int_{0}^{x} (R\ph\cdot \partial F(-\lm,\ph))\,\dy.\fish
\end{align*}
\end{equivenum}
\end{cor}

\begin{proof}
Since $\gradz F(\lm) = \frac{\gradz \Dz(\lm)}{\sqrt[c]{\Dz^{2}(\lm)-4}}$, the first claim follows from Theorem~\ref{thm:sym-dl} using that $\cph = R_{\al}\ph$ implies $\xi(x,-\lm,\ph) = -\xi(x,\lm,\ph)$. The second claim follows in analogous fashion using that $T\ph = \pm\ph$ implies $\grave{m}_{1}(-\lm,\ph) - \grave{m}_{4}(-\lm,\ph) = -(\grave{m}_{1}(\lm,\ph) - \grave{m}_{4}(\lm,\ph))$.\qed
\end{proof}

\section{Symmetries of the \nls  actions and \nls  hamiltonians}
\label{s:fin}

In this section we obtain several symmetries of the gradients of the \nls action variables which are subsequently used to prove Theorem~\ref{thm:bnf} (iii), Theorem~\ref{thm:nls-sym}, Theorem~\ref{thm:ham-mkdv-nls}, and Theorem~\ref{thm:freq}. Recall from~\eqref{Ink} that for $\ph\in\Wz$
the $n$th \nls action on level $k\in\Z$ is given by
\[
  I_{n,k}(\ph) = -\frac {1}{\pi}\int_{\Gmz_n(\ph)}\lm^{k-1} \Fz(\lm,\ph)\,\dlm,
\]
where $\Gmz_{n}(\ph)$ denotes a sufficiently close circuit around $G_{n}(\ph)$.

\begin{lem}
\label{sym-grad-In}
\begin{equivenum}
\item
On $\Wz$ we have for any $k\ge 1$, any $n\in\Z$, and any $\al\in\R$,
\[
  \gradz I_{n,k}(\ph)
   = (-1)^{k-1}P\gradz I_{-n,k}(\cph) 
   = (-1)^{k-1}T\gradz I_{-n,k}(T\ph) 
   = R_{\al}\gradz I_{n,k}(R_{\al}\ph).
\]
\item
If $\cph = R_{\al}\ph$ for some $\al\in\R$, then for any $k\ge 1$ and any $n\in\Z$,
\[
  \partial_{x}\partial I_{-n,k-1} - 2R\partial I_{-n,k}
  =
  (-1)^{k+1}
  \p*{\partial_{x}\partial I_{n,k-1} - 2R\partial I_{n,k}}.
\]
\item
If $T\ph = \pm\ph$, then for any $k\ge 1$ and any $n\in\Z$,
\begin{align*}
  &\partial_{x}\partial I_{-n,k-1} - 2R\partial I_{-n,k}
  - 2\ii \int_{0}^{x} (R\ph \cdot \partial I_{-n,k-1})\\
  &\qquad= (-1)^{k+1}
  \p*{\partial_{x}\partial I_{n,k-1} - 2R\partial I_{n,k}
  - 2\ii \int_{0}^{x} (R\ph \cdot \partial I_{n,k-1})}.\fish
\end{align*}
\end{equivenum}
\end{lem}

Item (i) for the case $k=1$ has been obtained in \cite{Grebert:2002wb}.

\begin{proof}
(i) By Theorem~\ref{thm:sym-dl}, $G_{n}(\ph) = -G_{-n}(\cph) = -G_{-n}(T\ph) = G_{n}(R_{\al}\ph)$ for any $n\in\Z$.
If, in addition $\ph\in \Wz$, then there exists a set of isolating neighborhoods $(U_{n})_{n\in\Z}$ which are mutually disjoint discs centered on the real axis such that $G_{n}\subset U_{n}$. They can be chosen such that $U_{n}(\ph) = -U_{-n}(\cph) = -U_{-n}(T\ph) = U_{n}(R_{\al}\ph)$ for all $n\in\Z$.
In particular, for any circuit $\Gmz_{n}(\ph)$ sufficiently close around $G_{n}(\ph)$, its inversion at the origin, $-\Gmz_{n}(\ph)$, defines a circuit around $\Gz_{-n}(\cph)$ of the same orientation as $\Gmz_{n}(\ph)$. Thus, with the substitution $\lm\mapsto -\lm$,
\begin{align*}
  I_{-n,k}(\cph) &= -\frac{1}{\pi}\int_{-\Gmz_{n}(\ph)} \lm^{k-1} F(\lm,\cph)\,\dlm\\
           &= \phantom{-}\frac{1}{\pi}\int_{\Gmz_{n}} (-\lm)^{k-1} F(-\lm,\cph)\,\dlm
           = (-1)^{k-1} I_{n,k}(\ph),
\end{align*}
where we used that $F(-\lm,\cph) = -F(\lm,\ph)$ by Corollary~\ref{sym-F}. Similarly, using
$F(-\lm,T\ph) = -F(\lm,\ph)$ one shows that $I_{-n,k}(T\ph) = I_{n,k}(\ph)$, and using
 $F(\lm,R_{\al}\ph) = F(\lm,\ph)$ one shows that $I_{n,k}(R_{\al}\ph) = I_{n,k}(\ph)$. Differentiating these identities gives
\begin{equation}
  \label{grad-In-sym}
  \gradz I_{n,k}(\ph) = (-1)^{k-1}P\gradz I_{-n,k}(\cph) = R_{\al}\gradz I_{n,k}(R_{\al}\ph).
\end{equation}

(ii) If $P\ph = R_{\al}\ph$, the periodic spectrum of $\ph$ is symmetric, and one has by Corollary~\ref{sym-grad-sym-F}
\begin{align*}
  \partial_{x}\partial I_{-n,k-1} - 2R\partial I_{-n,k}
  &= -\p*{
  -\frac{1}{\pi} \int_{\Gm_{n}} (-\lm)^{k-2} \p[\Big]{\partial_{x}\partial F(-\lm)
  + 2\lm R \partial F(-\lm)
  } \,\dlm
  }\\
  &= (-1)^{k+1}\p*{
  -\frac{1}{\pi} \int_{\Gm_{n}} \lm^{k-2} \p[\Big]{\partial_{x}\partial F(\lm)
  - 2\lm R \partial F(\lm)
  } \,\dlm
  }\\
  &= (-1)^{k+1}
  \p*{\partial_{x}\partial I_{n,k} - 2R\partial I_{n,k+1}}.
\end{align*}

(iii) If $T\ph = \pm\ph$, the periodic spectrum of $\ph$ is symmetric, and one has
\begin{align*}
  &\partial_{x}\partial I_{-n,k-1} - 2R\partial I_{-n,k}
    - 2\ii \int_{0}^{x} (R\ph \cdot \partial I_{-n,k-1})\\
   &\qquad = -\p*{
  -\frac{1}{\pi} \int_{\Gm_{n}} (-\lm)^{k-2} \p[\Big]{\partial_{x}\partial F(-\lm)
  + 2\lm R \partial F(-\lm)
  - 2\ii \int_{0}^{x} (R\ph \cdot \partial F(-\lm))
  } \,\dlm
  }\\
  &\qquad = (-1)^{k+1}
  \p*{\partial_{x}\partial I_{n,k-1} - 2R\partial I_{n,k}
  - 2\ii \int_{0}^{x} (R\ph \cdot \partial I_{n,k-1})}.\qed
\end{align*}
\end{proof}

\begin{proof}[Proof of Theorem~\ref{thm:bnf} (iii).]
Suppose $\ph_{u} \in \Wz$. Using that $PR = -\ii J$ we obtain from Lemma~\ref{sym-grad-In} (ii), applied in the case $k=1$,
\[
  -\ii J (\gradz I_{n} - \gradz I_{-n})\big|_{\ph_{u}}
   = \frac{1}{2} \partial_{x} P(\gradz I_{n,0} + P \gradz I_{n,0})\big|_{\ph_{u}}.
\]
Recall that $f^{\rest}(u) \defl f(\ph_{u})$ for any $C^{1}$-functional $f$ on $\Hc_{c}^{0}$, and
\[
  \gradh_{u} f^{\rest} = (\gradz f) \cdot (1,1)^{\top}\big|_{\ph_{u}} = (\gradz_{1} f + \gradz_{2} f)\big|_{\ph_{u}}.
\]
In particular,
\[
  -\ii J (\gradz I_{n} - \gradz I_{-n})\big|_{\ph_{u}}
   = \frac{1}{2}\partial_{x}\gradh_{u} I_{n,0}^{\rest} (1,1)^{\top},
\]
and in general for any $f$,
\begin{align*}
  \nbr{f,I_{n}-I_{-n}}\Big|_{\ph_{u}}
  &= -\ii\int_{\T} (\gradz f) \cdot J (\gradz I_{n,1} - \gradz I_{-n,1})\,\dx\bigg|_{\ph_{u}} \\
  &= \frac{1}{2}\int_{\T} (\gradh_{u} f^{\rest})\partial_{x}\gradh_{u} I_{n,0}^{\rest}\,\dx\bigg|_{u} \\
  &= \frac{1}{2}\gbr{f^{\rest},I_{n,0}^{\rest}} \bigg|_{u}
   = \gbr{f^{\rest},J_{n}} \bigg|_{u},
\end{align*}
where we used the identity $I_{n,0}^{\rest} = 2J_{n}$ from Lemma~\ref{comp-In-Jn} in the last step. 
Consequently, for $n,m\ge 1$,
\[
  \gbr{\th_{m}^{\rest},J_{n}}\Big|_{u} = \nbr{\th_{m},I_{n}-I_{-n}}\Big|_{\ph_{u}} = \dl_{m,n}.\qed
\]
\end{proof}

The next result implies Theorem~\ref{thm:nls-sym}.

\begin{prop}
\begin{equivenum}
\item
For every $\ph\in \Hc_{c}^{k-1}$ with $k\ge 1$, and any real $\al$,
\begin{align*}
  \gradz\Hz_{k}(\ph) &= (-1)^{k-1}P\gradz\Hz_{k}(\cph)
  = (-1)^{k-1}T\gradz\Hz_{k}(T\ph)
  = R_{\al}\gradz\Hz_{k}(R_{\al}\ph).
\end{align*}
\item
If $\ph\in \Hc_{c}^{2m-1}$, $m\ge 1$, with $\cph = R_{\al}\ph$ for some real $\al$, then
\[
  -\ii J \gradz \Hz_{2m}(\ph) = \partial_{x} \gradz \Hz_{2m-1}(\ph).
\]
\item
If $\ph\in \Hc_{c}^{2m-1}$, $m\ge 1$, with $T\ph = \pm\ph$, then
\[
  R(\gradz \Hz_{2m})
  =
  \partial_{x} (\gradz \Hz_{2m-1})
    - 2\ii \int_{0}^{x} (R\ph \cdot \partial \Hz_{2m-1})
  .\fish
\]
\end{equivenum}
\end{prop}

\begin{rem}
One verifies by direct computation that generically item (ii) and (iii) do not hold when $2m$ is replaced by $2m+1$ if $m\ge 2$.\map
\end{rem}

\begin{proof}
(i) On $\Wz\cap\Hc_{c}^{k-1}$ the sum $\sum_{n\in\Z} I_{n,k}$ converges locally uniformly to an analytic function  -- see \cite[Section 13]{Grebert:2014iq} -- and satisfies for $k\ge 1$
\[
  \sum_{n\in\Z} I_{n,k} = \frac{1}{2^{k-1}} \Hz_{k}.
\]
Since $\gradz I_{n,k}(\ph) = (-1)^{k-1}P\gradz I_{-n,k}(\cph) = (-1)^{k-1}T\gradz I_{-n,k}(T\ph) = R_{\al}\gradz I_{n,k}(R_{\al}\ph)$ by Lemma~\ref{sym-grad-In} (i), the first identity of Theorem~\ref{thm:nls-sym}, $\gradz\Hz_{k}(\ph) = (-1)^{k-1}P\gradz\Hz_{k}(\cph) = (-1)^{k-1}T\gradz\Hz_{k}(T\ph) = R_{\al}\gradz\Hz_{k}(R_{\al}\ph)$, follows for $\ph\in \Hc_{r}^{k-1}$. Since the Hamiltonians are analytic on $\Hc_{c}^{k-1}$, the identity extends to $\Hc_{c}^{k-1}$ by Lemma~\ref{ana-vanish}.

(ii) Suppose $\ph\in \Hc_{r}^{2k-1}$ with $\cph = R_{\al}\ph$ for some $\al\in\R$. Summing identity (ii) of Lemma~\ref{sym-grad-In} over $n\in\Z$ yields
\begin{align*}
  \frac{1}{2^{2m-1}} 2R(\gradz \Hz_{2m})
   &= \sum_{n\in\Z} 2R (\gradz I_{n,2m})
   = \sum_{n\in\Z} \partial_{x} (\gradz I_{n,2m-1})
    = \frac{1}{2^{2m-2}} \partial_{x} (\gradz \Hz_{2m-1}).
\end{align*}
Since both sides are analytic on $\Hc_{c}^{2k-1}$, the identities extend by Lemma~\ref{ana-vanish}. 

(iii) Similarly as for the previous item one obtains provided $T\ph = \pm \ph$,
\begin{align*}
  \frac{1}{2^{2m-1}} 2R(\gradz \Hz_{2m})
   &= \sum_{n\in\Z} 2R (\gradz I_{n,2m})
    = \sum_{n\in\Z} \p*{\partial_{x} (\gradz I_{n,2m-1})
    - 2\ii \int_{0}^{x} (R\ph \cdot \partial I_{n,2m-1})}\\
    &= \frac{1}{2^{2m-2}}\p*{\partial_{x} (\gradz \Hz_{2m-1})
    - 2\ii \int_{0}^{x} (R\ph \cdot \partial \Hz_{2m-1})}.\qed
\end{align*}
\end{proof}

\begin{proof}[Proof of Theorem~\ref{thm:ham-mkdv-nls}.]
At any point $\ph_{u}\in \Hc_{c}^{2m-1}$ we have by Theorem~\ref{thm:nls-sym} (ii)
\[
  X_{\Hz_{2m}}
   = -\ii J (\gradz \Hz_{2m})
   = \ii JR R(\gradz \Hz_{2m})
   = \partial_{x} P(\gradz \Hz_{2m-1}),\quad m\ge 1,
\]
where we used that $\ii JR = P$. Furthermore, $(\gradz \Hz_{2m-1}) = P (\gradz \Hz_{2m-1})$ by Theorem~\ref{thm:nls-sym} (i), hence
\[
  X_{\Hz_{2m}}^{\rest}
   = \frac{1}{2} (\partial_{x}(\gradz_{1}\Hz_{2m-1})^{\rest}
    + \partial_{x}(\gradz_{2}\Hz_{2m-1})^{\rest})(1,1)^{\top}
   = \frac{1}{2}(Y_{\Hz_{2m-1}^{\rest}},Y_{\Hz_{2m-1}^{\rest}})^{\top}.
\]
Since $\frac{1}{2}\Hz_{2m-1}^{\rest} = Y_{\Hh_{m}}$ by Proposition~\ref{ham-id-mkdv-nls}, it now follows that
\[
  X_{\Hz_{2m}}^{\rest} = (Y_{\Hh_{m}},Y_{\Hh_{m}}),\qquad m\ge 1.\qed
\]
\end{proof}

\begin{proof}[Proof of Theorem~\ref{thm:freq}.]
Suppose $J_{n}\neq 0$ and let $\phi_{t}$ be a local flow for the vector field $Y_{\Hh_{m}}$, then
\[
  \eta_{n,m} = -\gbr{\vt_{n},\Hh_{m}} = -\ddt\bigg|_{t=0} \vt_{n}\circ\phi_{t}.
\]
Since $\vt_{n}+\th_{n}^{\rest}$ is a function of the \mkdv actions only, we have
\begin{align*}
  -\ddt\bigg|_{t=0} \vt_{n}\circ\phi_{t} &= \ddt\bigg|_{t=0} \th_{n}^{\rest}\circ\phi_{t}.
\end{align*}
As $X_{\Hz_{2m}}^{\rest} = Y_{\Hh_{m}}$ by Theorem~\ref{thm:ham-mkdv-nls}, $\phi_{t}$ is also a local flow for the vector field $X_{\Hz_{2m}}^{\rest}$ and hence
\[
  \ddt\bigg|_{t=0} \th_{n}^{\rest}\circ\phi_{t} = \nbr{\th_{n},\Hz_{2m}} = \om_{n,2m}.\qed
\]
\end{proof}

\begin{appendix}

\section{Symmetries of the ZS fundamental solution}

Let $M(x,\lm,\ph) = \smat{m_{1} & m_{2}\\ m_{3} & m_{4}}$ denote the fundamental solution of the Zakharov-Shabat operator $L(\ph) = \smat{\ii & 0 \\ 0 & -\ii}\ddx + \smat{0 & \phm \\ \php & 0}$.
The following symmetries under the transformations $P$, $R_{\al}$, and $T$, introduced in Section~\ref{s:symmetries} have been noted in \cite{Grebert:2002wb}.

\begin{lem}
\label{F-sol-sym}
For any $x\in\R$, $\lm\in \C$, and $\ph\in \Hc_{r}^{0}$,
\begin{align*}
  \Mz(x,-\lm,\cph)
    &= \Jt \Mz(x,\lm,\ph)\Jt^{-1}
    = \mat[\bigg]{\phantom{-}m_4(x,\lm,\ph) & -m_3(x,\lm,\ph)\\ -m_2(x,\lm,\ph) & \phantom{-}m_1(x,\lm,\ph)},\\
    \Mz(x,\lm,R_{\al}\ph)
    &= R_{\al/2} \Mz(x,\lm,\ph)R_{\al/2}^{-1}
    = \mat[\bigg]{\phantom{\e^{-\ii\al}}m_1(x,\lm,\ph) & \e^{\ii\al}m_2(x,\lm,\ph)\\ \e^{-\ii\al}m_3(x,\lm,\ph) & \phantom{\e^{\ii\al}}m_4(x,\lm,\ph)},\\
    M(x,-\lm, T \ph) &= PJM(1-x,\lm,\ph)\grave{M}(\lm,\ph)^{-1}J^{-1}P.
\end{align*}
In particular,
\[
  \grave{M}(-\lm, T\ph) = \mat[\bigg]{\grave{m_4}(\lm,\ph) & \grave{m_2}(\lm,\ph)\\ \grave{m_3}(\lm,\ph) & \grave{m_1}(\lm,\ph)}.
\]
\end{lem}

\section{Analyticity}
\label{a:ana-maps}

\begin{lem}
\label{ana-vanish}
Let $X_r$ be an $\R$-Banach space and denote by $X$ its complexification. Assume that $U \subset X$ is an open connected neighborhood of $U_r = U \cap X_r$ and that $f\colon U \to \C$ is an analytic map. If $f\vert_{U_r} = 0$, then $f \equiv 0$.\fish
\end{lem}

\begin{proof}
Near any $u\in U_{r}$ the map $f$ is represented by its Taylor series,
\[
  f(u+h) = \sum_{n \ge 0} \frac{1}{n!} d_u^n f(h, \ldots, h),
\]
where the series converges absolutely and uniformly (cf. e.g. \cite[Theorem A.3]{Grebert:2014iq}). Since $f|_{U_r} = 0$, it follows that for any $h \in X_r$ and any $n\ge 0$, $d_u^n f(h, \ldots, h) = 0$. As $f$ is analytic, $d_u^n f$ is symmetric and $\C$-multilinear, hence it follows from the polarization identity that $d_u^n f(h, \ldots, h) = 0$ holds also for any $h$ in the complexification $X$ of $X_{r}$.
This implies $f \equiv 0$ in a neighborhood $V_u$ of $u$ with $V_u \subset U$. Since $U$ is connected it follows that $f\equiv 0$ on all of $U$ by the identity theorem.\qed
\end{proof}

\end{appendix}

\bibliography{bibcomplete}

\end{document}